\documentclass[11pt]{amsart}

\usepackage{amscd,amssymb,amsopn,amsmath,amsthm,mathrsfs,graphics,amsfonts,enumerate,verbatim,calc
}

\usepackage[all]{xy}

\usepackage{color}

\usepackage[OT2,OT1]{fontenc}
\newcommand\cyr{%
\renewcommand\rmdefault{wncyr}%
\renewcommand\sfdefault{wncyss}%
\renewcommand\encodingdefault{OT2}%
\normalfont
\selectfont}
\DeclareTextFontCommand{\textcyr}{\cyr}

\usepackage{amssymb,amsmath}

\DeclareFontFamily{OT1}{rsfs}{}
\DeclareFontShape{OT1}{rsfs}{n}{it}{<-> rsfs10}{}
\DeclareMathAlphabet{\mathscr}{OT1}{rsfs}{n}{it}

\numberwithin{equation}{section}
\newtheorem{theorem}{Theorem}[section]
\newtheorem{lemma}[theorem]{Lemma}
\newtheorem{proposition}[theorem]{Proposition}

\newtheorem{question}{Question}

\newtheorem*{maintheorem}{Main Theorem}

\newtheorem{maincorollary1}{Corollary}

\theoremstyle{definition}
\newtheorem{definition}[theorem]{Definition}
\newtheorem{remark}[theorem]{Remark}
\theoremstyle{remark}

\newtheorem{example}[theorem]{Example}
\newtheorem{acknowledgement}{Acknowledgement}

\renewcommand{\ker}{\operatorname{Ker}}

\newcommand{\fm}{\frak{m}}
\newcommand{\fp}{\frak{p}}
\newcommand{\fq}{\frak{q}}
\newcommand{\fa}{\frak{a}}

\begin{document}
\title[Tight closure of parameter ideals]
{Tight closure of parameter ideals in local rings $F$-rational on the punctured spectrum}

\author[P. H. Quy]{Pham Hung Quy}
\address{Department of Mathematics, FPT University, Hoa Lac Hi-Tech Park, Hanoi, Vietnam}
\email{quyph@fe.edu.vn}

\thanks{
The author is partially supported by a fund of Vietnam National Foundation for Science
and Technology Development (NAFOSTED) under grant number
101.04-2017.10.}

%\subjclass{13}
%\subjclass[2000]{Primary 13-XX}
%\subjclass[2000]{Primary ; Secondary}
%\date{\today \, (\printtime)}
%\date{\today}

\begin{abstract} Let $(R, \fm, k)$ be an equidimensional excellent local ring of characteristic $p>0$. The aim of this paper is to show that $\ell_R(\fq^*/\fq)$ does not depend on the choice of parameter ideal $\fq$ provided $R$ is an $F$-injective local ring that is $F$-rational on the punctured spectrum.
\end{abstract}

\maketitle

\section{Introduction}
Throughout this paper, let $(R, \fm, k)$ be a local ring of dimension $d$ and $\fq$ a parameter ideal of $R$. The motivation of this paper comes from the theory of Buchsbaum rings. Recall that the length $\ell_R (R/\fq)$ is always greater than or equal to the multiplicity $e(\fq)$ for all parameter ideals $\fq$. Furthermore, $R$ is Cohen-Macaulay if and only if $\ell_R (R/\fq) = e(\fq)$ for some (and hence for all) $\fq$. The ring $R$ is called {\it generalized Cohen-Macaulay} if the difference $\ell_R (R/\fq) - e(\fq)$ is bound above for all $\fq$. More precisely, if $R$ is generalized Cohen-Macaulay then
$$\ell_R (R/\fq) - e(\fq) \le \sum_{i=0}^{d-1} \binom{d-1}{i} \ell_R(H^i_{\fm}(R))$$
for all parameter ideals $\fq$, and the equality occurs for all parameter ideals $\fq$ contained in a large enough power of $\fm$. The ring $R$ is said to be {\it Buchsbaum} if  $\ell_R (R/\fq) - e(\fq)$ does not depend on the choice of parameter ideal $\fq$.\\

Suppose $(R, \fm, k)$ is an equidimensional excellent local ring of characteristic $p>0$. A classical result of Kunz \cite{Ku69} says that $R$ is regular if and only if the Frobenius endomorphism $F: R \to R, x \mapsto x^p$ is flat. Kunz's theorem is the starting point to study the singularities of $R$ in terms of Frobenius homomorphism, say {\it $F$-singularities}. $F$-singularities appear in the theory of \textit{tight closure} (cf. \cite{H96} for its introduction), which was systematically introduced by Hochster and Huneke around the mid 80's \cite{HH90}. The main objects of $F$-singularities are $F$-regularity, $F$-rationality, $F$-purity, and $F$-injectivity. Recall that $R$ is said to be {\it $F$-rational} if $\fq^* = \fq$ for all parameter ideals $\fq$, where $\fq^*$ is the tight closure of $\fq$. It should be noted that if $R$ is $F$-rational then it is Cohen-Macaulay and normal (here we assume that $R$ is excellent and equidimensional).
In \cite{GN02} Goto and Nakamura considered rings satisfying that $\ell_R (\fq^*/\fq)$ is bounded above. They proved that $\ell_R (\fq^*/\fq)$ is bounded above for every parameter ideal $\fq$ if and only if $R$ is $F$-rational on the punctured spectrum $\mathrm{Spec}^{\circ}(R) = \mathrm{Spec}(R) \setminus \{\fm \}$. It is worth to noting that these conditions imply the generalized Cohen-Macaulay property of $R$. In the present paper, we study the counterpart of Buchsbaum rings in the $F$-singularities realm. \\

Recall that the Frobenius endomorphism yields the (natural) Frobenius action on local cohomology $F: H^i_{\fm}(R) \to H^i_{\fm}(R)$ for all $i \ge 0$. We say $R$ is {\it $F$-injective} if the Frobenius action on all local cohomology modules $H^i_{\fm}(R)$ are injective (cf. \cite{F83}). Ma \cite{M15} showed that if an $F$-injective ring is generalized Cohen-Macaulay, then it is Buchsbaum. Therefore if $R$ is an $F$-injective local ring that is $F$-rational on the punctured spectrum, it is Buchsbaum. As the main result of this paper, we prove the following.
\begin{maintheorem}
  Let $(R, \fm, k)$ be an equidimensional excellent local ring of characteristic $p>0$. Suppose $R$ is an $F$-injective local ring that is $F$-rational on the punctured spectrum. Then $\ell_R(\fq^*/\fq)$ does not depend on the choice of parameter ideal $\fq$.
\end{maintheorem}
The paper is organized as follows. In the next section we collect results of generalized Cohen-Macaulay rings and of $F$-singularities used in this paper. Section 3 is devoted for a proof of the main theorem in the $F$-finite case. We prove the main theorem for any equidimensional excellent local ring in Section 4. In the last section we state several open questions for future research.
\begin{acknowledgement}
The author is deeply grateful to Linquan Ma for his discussion about the $\Gamma$-construction. He is also grateful to Nguyen Cong Minh and Kei-ichi Watanabe for their useful comments on this work. The author is grateful to the referee for his/her careful reading and useful comments.
\end{acknowledgement}

\section{Preliminary}

\subsection{Buchsbaum and generalized Cohen-Macaulay modules}

Let us recall the definition of Buchsbaum and generalized Cohen-Macaulay modules (cf. \cite{SV86,Tr86}). Let $M$ be a finitely generated module over a local ring $(R,\fm,k)$ and let $\fq$ be a parameter ideal of $M$. We denote by $e(\fq,M)$ the multiplicity of $M$ with respect to $\fq$ (cf. \cite{BH98} for details).

\begin{definition}
\label{nCM1}
Let $M$ be a finitely generated module over a Noetherian local ring $(R,\fm,k)$ such that $t=\dim M>0$. Then $M$ is called \textit{generalized Cohen-Macaulay}, if the difference
$$
\ell_R (M/\fq M)-e(\fq,M)
$$
is bounded above, where $\fq$ ranges over the set of all parameter ideals of $M$.
\end{definition}

\begin{remark}
Let the notation be as in Definition \ref{nCM1}.
\begin{enumerate}
\item It is well-known that $M$ is Cohen-Macaulay if and only if $H^i_{\fm}(M) = 0$ for all $i<t$. Similarly, $M$ is generalized Cohen-Macaulay if and only if $H^i_{\fm}(M)$ is a finitely generated $R$-module for all $i < t$.

\item
Suppose $R$ is equidimensional and is an image of a Cohen-Macaulay local ring. Then $M$ is generalized Cohen-Macaulay if and only if the non-Cohen-Macaulay locus of $M$ is isolated.

\item
Let $M$ be a generalized Cohen-Macaulay $R$-module over $(R,\fm,k)$ such that $t=\dim M>0$. Then
$$
\ell_R (M/\fq M)-e(\fq, M) \le \sum_{i=0}^{t-1} \binom{t-1}{i} \ell_R(H^i_{\fm}(M))
$$
for every parameter ideal $\fq$ of $M$.
\end{enumerate}
\end{remark}

\begin{definition}[cf. \cite{Tr86}]
\label{Buchsbaum}
Let $M$ be a finitely generated module over a Noetherian local ring $(R,\fm,k)$ such that $t=\dim M>0$. A parameter ideal $\fq$ of $M$ is called \textit{standard} if
$$
\ell_R(M/\fq M)-e(\fq, M) = \sum_{i=0}^{t-1} \binom{t-1}{i} \ell_R(H^i_{\fm}(M)).
$$
An $\fm$-primary ideal $\fa$ is said to be {\it standard} if every parameter ideal contained in $\fa$ is standard.
We say that $M$ is \textit{Buchsbaum} if every parameter ideal of $M$ is standard, i.e. $\fm $ is a standard ideal of $M$.
\end{definition}

Standard ideals admit a cohomological characterization as follows.
\begin{remark}\label{standard ideal}
\begin{enumerate}
\item The parameter ideal $\fq$ of $M$ is standard if and only if the canonical homomorphism $H^i(\fq; M) \to H^i_{\fm}(M)$ is surjective for all $i<t$. Moreover $M$ is Buchsbaum if and only if the canonical homomorphism $H^i(\fm; M) \to H^i_{\fm}(M)$ is surjective for all $i<t$, where the Koszul cohomology can be defined by any set of generators of $\fm$ (\cite[Theorem 3.4]{Tr86}).
\item Let $M$ be a generalized Cohen-Macaulay module over $(R,\fm,k)$ such that $d=\dim M>0$ and let $n \in \mathbb{N}$ be a positive integer such that $\fm^n \, H^i_{\fm}(M) = 0$ for all $i<d$. Then every parameter element $x \in \fm^{2n}$ of $M$ admits the splitting property, i.e. $H^i_{\fm}(M/xM) \cong H^i_{\fm}(M) \oplus H^{i+1}_{\fm}(M)$ for all $i<d-1$. Furthermore, every parameter ideal contained in $\fm^{2n}$ is standard (cf. \cite{CQ11}).
\end{enumerate}
\end{remark}
We also need the notion of limit closure of parameter ideal in the sequel.
\begin{definition}
\label{limit}
Let $(R,\fm,k)$ be a local ring, let $M$ be a finitely generated module with $t=\dim M$ and let $\underline{x} =x_1,\ldots,x_t$ be a of system of parameters of $M$. The \textit{limit closure} of $\underline{x}$ in $M$ is defined as a submodule of $M$:
$$
(\underline{x})_M^{\lim}=\bigcup_{n>0}\big((x_1^{n+1},\ldots,x_t^{n+1})M:_M (x_1 \cdots x_t)^n \big)
$$
with the convention that $(\underline{x})_M^{\lim}=0$ when $t=0$. If $M=R$, then we simply write $(\underline{x})^{\lim}$.
\end{definition}

From the definition, it is clear that $(\underline{x})M \subseteq (\underline{x})_M^{\lim}$.

\begin{remark}
\label{limitclosure}
Let the notation be as in Definition \ref{limit}.
\begin{enumerate}
\item
The quotient $(\underline{x})_M^{\lim}/(\underline{x})M$ is the kernel of the canonical map
$$H^t(\underline{x};M) \cong M/(\underline{x})M \to H^t_{\fm}(M).$$ This implies the following fact. Let $\fq=(x_1,\ldots,x_t)$ and put $\fq_M^{\lim}:=(\underline{x})^{\lim}_M$. Hence the notation $\fq_M^{\lim}$ is independent of the choice of $x_1, \ldots,x_t$ which generate $\fq$.

\item
It is known that $(\underline{x})M=(\underline{x})_M^{\lim}$ if and only if $\underline{x}$ forms an $M$-regular sequence.

\item
It is shown that the Hochster's monomial conjecture is equivalent to the claim that $\fq^{\lim} \neq R$ for every parameter ideal $\fq$ of $R$.
\item If $M$ is a generalized Cohen-Macaulay module of dimension $t>0$, then for every parameter ideal $\fq$ we have
$$ \ell_R(\fq^{\lim}_M/ \fq M ) \le \sum_{i=0}^{t-1}\binom{t}{i} \ell_R (H^i_{\fm}(M)),$$
and the equality occurs if and only if $\fq$ is standard by Definition \ref{Buchsbaum} and \cite[Theorem 5.1]{CHL99} (see also \cite[Proposition 3.6]{G83}).
\end{enumerate}
\end{remark}

\subsection{Tight closure and $F$-singularities} In the rest of this paper, we always assume that $(R, \fm, k)$ is an excellent equidimnesional local ring of characteristic $p>0$ and of dimension $d>0$. If we want to notationally distinguish the source and target of the $e$-th Frobenius endomorphism $F^e: R \overset{x \mapsto x^{p^e}}{\longrightarrow} R$, we will use $F_*^e(R)$ to denote the target. $F_*^e(R)$ is an $R$-bimodule, which is the same as $R$ as
an abelian group and as a right $R$-module, that acquires its left $R$-module structure via the $e$-th Frobenius
endomorphism $F^e$. By definition the $e$-th Frobenius endomorphism $F^e: R \to F_*^e(R)$ sending $x$ to $F_*^e(x^{p^e}) = x \cdot F_*^e(1)$ is an $R$-homomorphism. We say $R$ is {\it $F$-finite} if $F_*(R)$ is a finite $R$-module. When $R$ is reduced, we will use $R^{1/p^e}$ to denote the ring whose elements are $p^e$-th roots of elements of $R$. Note that these notations (when $R$ is reduced) $F_*^e(R)$ and $R^{1/p^e}$ are used interchangeably in the literature.
\begin{definition}[\cite{HH90,HH94,H96}] Let $R^{\circ} = R \setminus \cup_{\fp \in \mathrm{Min}R} \fp$. Then for any ideal $I$ of $R$ we define
\begin{enumerate}
  \item The {\it Frobenius closure} of $I$, $I^F = \{x \mid  x^q \in I^{[q]} \text{ for some } q = p^e\}$, where $I^{[q]} = (x^q \mid x \in I)$.
  \item The {\it tight closure} of $I$, $I^* = \{x \mid cx^q \in I^{[q]} \text{ for some } c \in R^{\circ} \text{ and for all } q \gg 0\}$.
\end{enumerate}

\end{definition}

\begin{remark} An element $x \in I^F$ if it is contained in the kernel of the composition
 $$R \to R/I \overset{\mathrm{id} \otimes F^e}{ \longrightarrow} R/I \otimes F_*^e(R)$$
 for some $e \ge 0$. Similarly, an element $x \in I^*$ if it is contained in the kernel of the composition
$$R \to R/I \to R/I \otimes F_*^e(R) \overset{\mathrm{id} \otimes F_*^e(c)}{\longrightarrow} R/I \otimes F_*^e(R)$$ for some $c \in R^{\circ}$ and for all $e \gg 0$. In general, let $N$ be a submodule of an $R$-module $M$. The tight closure of $N$ in $M$, denoted by $N^*_M$, consists elements that are contained in the kernel of the composition
 $$M \to M/N \to M/N \otimes F_*^e(R) \overset{\mathrm{id} \otimes F_*^e(c)}{\longrightarrow} M/N \otimes F_*^e(R)$$ for some $c \in R^{\circ}$ and for all $e \gg 0$.
\end{remark}

Let $x_1, \ldots, x_d$ be a system of parameters of $R$. Recall that local cohomology $H^i_{\fm}(R)$ may be computed as the homology of the \v{C}ech complex
$$
0 \to R \to \bigoplus_{i=1}^t R_{x_i} \to \cdots \to R_{x_1 \ldots x_d} \to 0.
$$
Then the Frobenius endomorphism $F:R \to R$ induces a natural Frobenius action $F:H^i_{\fm}(R) \to H^i_{\fm^{[p]}}(R) \cong H^i_{\fm}(R)$. There is a very useful way of describing the top local cohomology. It can be given as the direct limit of Koszul cohomologies
$$
H^d_{\fm}(R) \cong \lim_{\longrightarrow} R/(x_1^n, \ldots, x_d^n).
$$
Then for each $\overline{a} \in H^d_{\fm}(R)$, which is the canonical image of $a+(x_1^n, \ldots, x_d^n)$, we find that $F(\overline{a})$ is the canonical image of $a^p +(x_1^{pn}, \ldots, x_d^{pn})$ (see \cite{Sm97}).

\begin{remark}\label{direct system} Recall that we always assume $R$ is excellent and equidimensional.
 \begin{enumerate}
     \item An element $x \in 0^*_{H^d_{\fm}(R)}$ if there exists $c \in R^{\circ}$ such that $c \, F^e(x) = 0$ for all $e \gg 0$. Let $x_1, \ldots, x_d$ be a system of parameters of $R$. The direct system $\lim  R/(x_1^n, \ldots, x_d^n) \cong H^d_{\fm}(R)$ induces the direct system of tight closures
         $$\lim_{\longrightarrow}\, (x_1^n, \ldots, x_d^n)^*/(x_1^n, \ldots, x_d^n) \cong 0^*_{H^d_{\fm}(R)}.$$
         By \cite[Remark 5.4]{Hu98} we have $(x_1, \ldots, x_d)^{\lim} \subseteq (x_1, \ldots, x_d)^*$ for all system of parameters $x_1, \ldots, x_d$. By Remark \ref{limitclosure} (1) we obtain the direct system
         $$\lim_{\longrightarrow}\, (x_1^n, \ldots, x_d^n)^*/(x_1^n, \ldots, x_d^n)^{\lim} \cong 0^*_{H^d_{\fm}(R)}$$
         with all maps in the direct system are injective. As a consequence, $\ell_R \, (\fq^*/ \fq^{\lim})$ is bounded above if and only if $\ell_R (0^*_{H^d_{\fm}(R)}) < \infty$. In this case
         $\max_{\fq} \{\ell_R \, (\fq^*/ \fq^{\lim})\} = \ell_R (0^*_{H^d_{\fm}(R)})$, and $\ell_R \, (\fq^*/ \fq^{\lim}) = \ell_R (0^*_{H^d_{\fm}(R)})$ for all parameter ideals $\fq$ contained in a large enough power of maximal ideal.
%     \item For $i<d$ we note that $\dim R/\mathrm{Ann}\, H^i_{\fm}(R) \le i <d$ (cf. \cite[Theorem 8.1.1 (b)]{BH98}). By choosing %$c \in R^{\circ} \cap \mathrm{Ann}\, H^i_{\fm}(R)$ we have $0^*_{H^i_{\fm}(R)} = H^i_{\fm}(R)$.
     \item If we consider the target of the Frobenius endomorphism $F: R \to R$ as an $R$-module $F_*(R)$ via the Frobebenius endomorphism, then the Frobenius action on $H^i_{\fm}(R)$ becomes an $R$-homomorphism $F: H^i_{\fm}(R) \to H^i_{\fm}(F_*(R))$ for all $i \ge 0$.
   \end{enumerate}
\end{remark}
We now present $F$-singularities used in this paper.
\begin{definition} A local ring $(R, \fm, k)$ is called {\it $F$-rational} if every parameter ideal is tight closed, i.e. $\fq^* = \fq$ for all $\fq$.
\end{definition}
\begin{remark} It is well known that an equidimensional excellent local ring $R$ is $F$-rational if and only if $R$ is Cohen-Macualay and $0^*_{H^d_{\fm}(R)} = 0$. Furthermore, $R$ is normal provided it is $F$-rational.
\end{remark}
The following is the main result of Goto and Nakamura \cite[Theorem 1.1]{GN02}.
\begin{theorem} \label{Goto} Let $(R, \fm, k)$ be an equidimensional excellent local ring. Then the following are equivalent.
\begin{enumerate}
  \item $\ell_R (\fq^* / \fq)$ is bounded above;
  \item $R$ is $F$-rational on the punctured spectrum $\mathrm{Spec}^{\circ}(R) = \mathrm{Spec}(R) \setminus \{ \fm \}$;
  \item $R$ is generalized Cohen-Macaulay and $\ell_R (0^*_{H^d_{\fm}(R)}) < \infty$.
\end{enumerate}

\end{theorem}

In more detail, we have the following.
\begin{proposition} \label{upper above} Let $(R, \fm, k)$ be an equidimensional excellent local ring of characteristic $p>0$ and of dimension $d>0$ that is $F$-rational on the punctured spectrum. Then for every parameter ideal $\fq$ of $R$ we have
$$\ell_R\, (\fq^* / \fq) \le \sum_{i=0}^{d-1}\binom{d}{i} \ell_R (H^i_{\fm}(R)) + \ell_R (0^*_{H^d_{\fm}(R)}).$$
Moreover, the equality occurs if and only if $\fq$ is a standard parameter ideal satisfying one of the following condition
\begin{enumerate}
  \item The canonical map $\fq^* / \fq \to 0^*_{H^d_{\fm}(R)}$ is surjective;
  \item The canonical map $\fq^* / \fq^{\lim} \to 0^*_{H^d_{\fm}(R)}$ is an isomorphism.
\end{enumerate}
Furthermore, we have the equality for all parameter ideals $\fq$ contained in a large enough power of $\fm$.
\end{proposition}
\begin{proof} Note that $\fq  \subseteq \fq^{\lim}  \subseteq\fq^*$, so $\ell_R\, (\fq^* / \fq) = \ell_R\, (\fq^* / \fq^{\lim}) + \ell_R\, (\fq^{\lim} / \fq)$.
The assertion now follows from Remarks \ref{limitclosure} (4) and  \ref{direct system} (1).
\end{proof}

Since every parameter ideal of a Buchsbaum ring is standard, we have
\begin{proposition} \label{constant} Let $(R, \fm, k)$ be an equidimensional excellent local ring of characteristic $p>0$ and of dimension $d>0$. Suppose $R$ is a Buchsbaum ring that is $F$-rational on the punctured spectrum. Then the following are equivalent.
\begin{enumerate}
  \item $\ell_R\, (\fq^* / \fq) = \sum_{i=0}^{d-1}\binom{d}{i} \ell_R (H^i_{\fm}(R)) + \ell_R (0^*_{H^d_{\fm}(R)})$ for all parameter ideals $\fq$;
  \item The canonical map $\fq^* / \fq \to 0^*_{H^d_{\fm}(R)}$ is surjective for all parameter ideals $\fq$;
  \item The canonical map $\fq^* / \fq^{\lim} \to 0^*_{H^d_{\fm}(R)}$ is an isomorphism for all parameter ideals $\fq$.
\end{enumerate}
\end{proposition}

The relation between Buchsbaum rings and $F$-singularities appears explicit in \cite{M15}.
\begin{definition} A local ring $(R,\fm,k)$ is \textit{$F$-injective} if the Frobenius action on $H^i_{\fm}(R)$ is injective for all $i \ge 0$.
\end{definition}

\begin{remark}\label{injective} If $R$ is $F$-injective then it is reduced (cf. \cite[Lemma 3.11]{QS17}). Conversely, a reduced ring $R$ is $F$-injective if and only if the inclusion $R \hookrightarrow F_*^e(R)$ induces injective $R$-homomorphisms $H^i_{\fm}(R) \to H^i_{\fm}(F_*^e(R))$ for all $e, i \ge 0$.

\end{remark}

\begin{remark} Recently, Ma \cite[Corollary 3.5]{M15} showed that an $F$-injective generalized Cohen-Macaulay ring is Buchsbaum. Thus we have the following implications of singularities in this paper
\[
\xymatrix{
 \text{ F\text{-rational} }\ar@{=>}[d] \ar@{=>}[r] & F\text{-rational/} \mathrm{Spec}^{\circ}(R)\, \& \,F\text{-injective}  \ar@{=>}[d] \ar@{=>}[r] & F\text{-rational/} \mathrm{Spec}^{\circ}(R) \ar@{=>}[d]\\
 \text{CM \& normal} \ar@{=>}[r]          & \text{Buchsbaum \& reduced} \ar@{=>}[r]  & \text{Generalized CM.}
          }
\]

\end{remark}
\section{Proof of the main theorem in the $F$-finite case}
We prove the main theorem for $F$-finite rings. Our method is inspired by the proof of Theorem 3.7 of \cite[Arxiv: 1512.05374, Version 1]{BMS16}. Note that any $F$-finite ring is excellent by Kunz \cite{Ku76}.

\begin{theorem}\label{F-finite case}
  Let $(R, \fm)$ be an equidimensional $F$-finite local ring of dimension $d>0$. Suppose $R$ is an $F$-injective local ring that is $F$-rational on the punctured spectrum. Then $\ell_R(\fq^*/\fq)$ does not depend on the choice of parameter ideal $\fq$.
\end{theorem}
\begin{proof} By \cite[Corollary 3.5]{M15} we have $R$ is Buchsbaum. By Proposition \ref{constant} we need only to show that the canonical map $\fq^* / \fq^{\lim} \to 0^*_{H^d_{\fm}(R)}$ is an isomorphism for any parameter ideal $\fq$. Let $e$ be a large enough positive integer such that canonical map $(\fq^{[p^e]})^* / (\fq^{[p^e]})^{\lim} \to 0^*_{H^d_{\fm}(R)}$ is an isomorphism (cf. Remark \ref{direct system}). Since the role of $\fq^{[p^e]}$ in $R$ is the same as that of $\fq$ in $F_*^e(R)$, so the canonical map
$$\fq_ {F_*^e(R)}^* / \fq^{\lim}_{F_*^e(R)} \to 0^*_{H^d_{\fm}(F_*^e(R))}$$ is an isomorphism, where $\fq_ {F_*^e(R)}^*$ is the tight closure of $\fq \, F_*^e(R)$ as a submodule of $F_*^e(R)$. Since $R$ is reduced, we have the short exact sequence of $R$-module
$$0 \to R \to F_*^e(R) \to F_*^e(R)/R \to 0,$$
where both $F_*^e(R)$ and $F_*^e(R)/R$ are Buchsbaum as $R$-modules (see the proofs of \cite[Theorem 3.4]{M15} and \cite[Theorem 4.17]{QS17}). Set $S = F_*^e(R)$. By Remark \ref{injective} we have the short exact sequence of local cohomology
$$0 \to H^i_\fm(R) \to H^i_\fm(S) \to H^i_\fm(S/R) \to 0$$
for all $i \ge 0$. Thus
$$\ell_R \, (H^i_\fm(S)) = \ell_R \, (H^i_\fm(R)) + \ell_R \, (H^i_\fm(S/R))$$
for all $0 \le i \le d-1$. By the proof of \cite[Theorem 4.17]{QS17} we have the following short exact sequence
$$0 \to R/ \fq \to S/ \fq S \to (S/R)/ \fq (S/R) \to 0.$$
Therefore we have the following commutative diagram
$$
\begin{CD}
0 @>>>  R/ \fq  @>>> S/ \fq S @>>> (S/R)/ \fq (S/R) @>>> 0 \\
@. @V\beta VV @V\alpha VV @V\varphi VV \\
0 @>>> H^d_\fm(R) @>>> H^i_\fm(S) @>>> H^i_\fm(S/R) @>>>0.
\end{CD}
$$
By Remark \ref{direct system} (1), $\ker (\beta) = \fq^{\lim} / \fq$. Since $R$ is Buchsbaum,
$$\ell_R \, (\ker (\beta)) = \sum_{i=0}^{d-1} \binom{d}{i} \ell_R \, (H^i_{\fm}(R)).$$
Similarly, we have
$$\ell_R \, (\ker (\alpha)) = \sum_{i=0}^{d-1} \binom{d}{i} \ell_R \, (H^i_{\fm}(S)),$$
and
$$\ell_R \, (\ker (\varphi)) = \sum_{i=0}^{d-1} \binom{d}{i} \ell_R \, (H^i_{\fm}(S/R)).$$
Notice that $\ell_R \, (H^i_\fm(S)) = \ell_R \, (H^i_\fm(R)) + \ell_R \, (H^i_\fm(S/R))$ for all $0 \le i \le d-1$, so
$$\ell_R \, (\ker (\beta)) =  \ell_R \, (\ker (\alpha)) + \ell_R \, (\ker (\varphi)).$$
Therefore the above commutative diagram induces the following commutative diagram
$$
\begin{CD}
0 @>>>  R/ \fq^{\lim}  @>\nu >> S/ \fq^{\lim}_S @>\pi >> (S/R)/ \fq^{\lim}_{(S/R)} @>>> 0 \\
@. @V\overline{\beta} VV @V\overline{\alpha} VV @V\overline{\varphi} VV \\
0 @>>> H^d_\fm(R) @>\mu >> H^i_\fm(S) @>\theta >> H^i_\fm(S/R) @>>>0,
\end{CD}
$$
where $\overline{\alpha}, \overline{\beta}$ and $\overline{\varphi}$ are injective. As above we known that the restriction of $\overline{\alpha}$ on the tight closures
$\fq_ S^* / \fq^{\lim}_S \to 0^*_{H^d_{\fm}(S)}$ is an isomorphism. We claim the restriction map
$$\overline{\beta}: \fq^* / \fq^{\lim} \to 0^*_{H^d_{\fm}(R)}$$
is an isomorphism.  It is enough to show this map is surjective. Indeed, let $y$ be any element in $0^*_{H^d_{\fm}(R)}$. We have $\mu(y) \in 0^*_{H^d_{\fm}(S)}$. Thus there exists $z \in \fq^*_S$ such that $\overline{\alpha} (\overline{z}) = \mu(y)$. We have
 $$\overline{\varphi} (\pi (\overline{z})) = \theta ( \overline{\alpha} (\overline{z}))= \theta ( \mu (y)) = 0.$$
 Thus $\pi (\overline{z}) = 0$ since $\overline{\varphi}$ is injective. Therefore we have $x \in R$ such that $\nu (\overline{x}) = \overline{z} \in \fq_ S^* / \fq^{\lim}_S$. Note that the role of $\fq$ in $F_*(R)$ is the same that of $\fq^{[p^e]}$ in $R$. Now the condition of $x$ means $x^{p^e} \in (\fq^{[p^e]})^*$. Thus we have some $c \in R^{\circ}$ such that $c x^{p^{e+e'}} \in \fq^{[p^{e+e'}]}$ for all $e' \gg 0$. Hence $x \in \fq^*$. Moreover $\mu(\overline{\beta}(\overline{x})) = \overline{\alpha} (\nu (x)) = \mu(y)$. Thus $\overline{\beta}(\overline{x})  = y$ since $\mu$ is injective, and so the canonical map $\fq^* / \fq^{\lim} \to 0^*_{H^d_{\fm}(R)}$ is surjective as desired. The proof is complete.
\end{proof}

We do not have the converse of the previous theorem by the following.

\begin{example} Let $R = k[X, Y, Z]/(X^3 + Y^3 + Z^3)$, where $k$ is a perfect field of characteristic $0< p \neq 3$. The ring $R$ is a Gorenstein ring with isolated singular. The test ideal of $R$ is the maximal ideal (cf. \cite[Example 4.8]{Hu98}). We have $\ell_R \, (\fq^* / \fq) = 1$ for all parameter ideals $\fq$. However $R$ is $F$-injective if and only if $p \equiv 1 \, \mathrm{mod} \, 3$.
\end{example}

\section{Proof of the main theorem for excellent local rings}
We will prove the main theorem mentioned in the introduction. The key ingredient is using the $\Gamma$-construction of Hochster and Huneke \cite{HH94} to reduce the problem to the case of $F$-finite rings of the previous section. We briefly recall the construction. Let $(R,\fm,k)$ be a complete local ring with coefficient field $k$ of characteristic $p>0$ and of dimension $d$. Let us fix a $p$-basis $\Lambda$ of $k \subset R$ and let $\Gamma \subset \Lambda$ be a cofinite subset (we refer the reader to \cite{Ma86} for the definition of a $p$-basis). We denote by $k_e$ (or $k_{\Gamma,e}$ to signify the dependence on the choice of $\Gamma$) the purely inseparable  extension field $k[\Gamma^{1/p^e}]$ of $k$, which is obtained by adjoining $p^e$-th roots of all elements in $\Gamma$. Next, fix a system of parameters $x_1,\ldots,x_d$ of $R$. Then the natural map $A:=k[[x_1,\ldots,x_d]] \to R$ is module-finite. Let us put
$$
A^{\Gamma}:=\bigcup_{e>0} k_e[[x_1,\ldots,x_d]].
$$
Then the natural map $A \to A^{\Gamma}$ is faithfully flat and purely inseparable and $\fm_A A^{\Gamma}$ is the unique maximal ideal of $A^{\Gamma}$. Now we set $R^{\Gamma}:=A^{\Gamma} \otimes_A R$. Then $R \to R^{\Gamma}$ is faithfully flat and purely inseparable. The maximal ideal of $R$ expands to the maximal ideal of $R^{\Gamma}$. The crucial fact about $R^{\Gamma}$ is that it is an $F$-finite local ring (see \cite[(6.6) Lemma]{HH94}). Moreover, we can preserve some good properties of $R$
if we choose a sufficiently small cofinite subset $\Gamma$. For example, if $R$ is reduced, so is $R^{\Gamma}$ for any sufficiently small choice of $\Gamma$ cofinite in $\Lambda$. For each prime ideal $\fp$ of $R$, $\sqrt{\fp R^{\Gamma}}$ is a prime ideal of $R^{\Gamma}$. Furthermore if we choose $\Gamma$ small enough, $\fp R^{\Gamma}$ is also prime. We need several lemmas.

\begin{lemma}[\cite{EH08}, Lemma 2.9] \label{Gamma injec} Let $(R, \fm, k)$ be a complete local ring that is $F$-injective. Then for any sufficiently small choice of $\Gamma$ cofinite in $\Lambda$, $R^{\Gamma}$ is $F$-injective.
\end{lemma}

The next result can be proven by the same method used in \cite[Proposition 5.6]{M14}, so we omit the detail proof. Note that we use \cite[Lemma 2.3, Theorem 3.5]{V95} to replace the roles of \cite[Proposition 5.4, Lemma 5.5]{M14} in the proof of \cite[Proposition 5.6]{M14}.

\begin{lemma}\label{Gamma F-rational} Let $(R, \fm, k)$ be a complete local ring that is $F$-rational on the punctured spectrum. Then for any sufficiently small choice of $\Gamma$ cofinite in $\Lambda$, $R^{\Gamma}$ is $F$-rational on the punctured spectrum.
\end{lemma}
Let $M$ be an $R$-module with a Frobenius action $F$. A submodule $N$ of $M$ is called \textit{F-compatible} if $F(N) \subseteq N$. An $R$-module $M$ with a Frobenius action is said to be {\it simple} if it has no nontrivial $F$-compatible submodules. If $R$ is a complete local domain, $0^*_{H^d_{\fm}(R)}$ is the unique maximal proper $F$-compatible submodule of $H^d_{\fm}(R)$ by Smith \cite{Sm97}. Hence $H^d_{\fm}(R)/0^*_{H^d_{\fm}(R)}$ is simple. In general, for reduced equidimensional excellent local rings Smith \cite[Proposition 2.5]{Sm97} showed that $0^*_{H^d_{\fm}(R)}$ is the unique maximal proper $F$-compatible submodule of $H^d_{\fm}(R)$ which is annihilated by an element $c \in R^{\circ}$.

\begin{lemma}\label{Gamma local coho}
  Let $(R, \fm, k)$ be an equidimensional complete local ring. Then for any sufficiently small choice of $\Gamma$ cofinite in $\Lambda$, $0^*_{H^d_{\fm}(R^{\Gamma})} = 0^*_{H^d_{\fm}(R)} \otimes_R R^{\Gamma}$.
\end{lemma}
\begin{proof} It is obvious that we can assume $R$ is reduced. Set $(0) = \fp_1 \cap \cdots \cap \fp_n$ the primary decomposition of $(0)$. We prove by induction on $n$ the following claim.\\
\noindent {\bf Claim.} $H^d_{\fm}(R)/0^*_{H^d_{\fm}(R)} \cong H^d_{\fm}(R_1)/0^*_{H^d_{\fm}(R_1)} \oplus \cdots \oplus H^d_{\fm}(R_n)/0^*_{H^d_{\fm}(R_n)}$ as $R$-modules with Frobenius actions, where $R_i = R/\fp_i$ for all $1 \le i \le n$.\\
Indeed, we have nothing to do when $n =1$. Suppose $n>1$ and set $R' = R/(\fp_1 \cap \cdots \cap \fp_{n-1})$. We have the following short exact sequence
$$0 \to R \to R' \oplus R_n \to R'' \to 0,$$
where $R'' = R/(\fp_1 \cap \cdots \cap \fp_{n-1} + \fp_n)$. This short exact sequence induces the following exact sequence of local cohomology with homomorphisms are compatible with Frobenius actions
$$\cdots \to H^{d-1}_{\fm}(R'') \overset{\delta}{ \to } H^d_{\fm}(R) \to H^d_{\fm}(R') \oplus H^d_{\fm}(R_n) \to H^d_{\fm}(R'') = 0.$$
Note that $\dim R'' < d$, so $\mathrm{Im} (\delta)$ is an $F$-compatible submodule of $H^d_{\fm}(R)$ of dimension less than $d$, and so that $\mathrm{Im} (\delta) \subseteq 0^*_{H^d_{\fm}(R)}$. Therefore we obtain the following short exact sequence by restricting on tight closures
$$0 \to \mathrm{Im} (\delta) \to 0^*_{H^d_{\fm}(R)} \to 0^*_{H^d_{\fm}(R')} \oplus 0^*_{H^d_{\fm}(R_n)} \to 0.$$
Thus we have the following commutative diagram
$$
\begin{CD}
0 @>>> \mathrm{Im} (\delta)   @>>> 0^*_{H^d_{\fm}(R)} @>>> 0^*_{H^d_{\fm}(R')} \oplus 0^*_{H^d_{\fm}(R_n)} @>>> 0 \\
@. @VVV @VVV @VVV \\
0 @>>>\mathrm{Im} (\delta) @>>>  H^d_{\fm}(R) @>>> H^d_{\fm}(R') \oplus H^d_{\fm}(R_n) @>>>0.
\end{CD}
$$
Hence $H^d_{\fm}(R)/0^*_{H^d_{\fm}(R)} \cong H^d_{\fm}(R')/0^*_{H^d_{\fm}(R')} \oplus H^d_{\fm}(R_n)/0^*_{H^d_{\fm}(R_n)}$. The Claim now follows from the inductive hypothesis.\\

We continue to prove the lemma. Notice that for all $i \le n$, $H^d_{\fm}(R_i)/0^*_{H^d_{\fm}(R_i)}$ is simple by Smith's result. For any sufficiently small choice of $\Gamma$ cofinite in $\Lambda$ we have $R^{\Gamma}$ is reduced (so $\fp_i R^{\Gamma}$ is prime for all $i \le n$). We can even assume that $(H^d_{\fm}(R_i)/0^*_{H^d_{\fm}(R_i)}) \otimes_R R^{\Gamma}$ is simple as an $R^{\Gamma}$-module with a Frobenius action for all $i \le n$ by \cite[Lemma 5.8]{M14}. Thus
$$(H^d_{\fm}(R)/0^*_{H^d_{\fm}(R)})  \otimes_R R^{\Gamma} \cong H^d_{\fm}(R^{\Gamma})/(0^*_{H^d_{\fm}(R)}\otimes_R R^{\Gamma})$$ is a direct sum of $n$ simple $R^{\Gamma}$-modules with Frobenius actions. On the other hand $0^*_{H^d_{\fm}(R)}\otimes_R R^{\Gamma} \subseteq 0^*_{H^d_{\fm}(R^{\Gamma})}$ and $H^d_{\fm}(R^{\Gamma})/0^*_{H^d_{\fm}(R^{\Gamma})}$ is also a direct sum of $n$ simple $R^{\Gamma}$-modules with Frobenius actions by the Claim. Therefore $0^*_{H^d_{\fm}(R^{\Gamma})} = 0^*_{H^d_{\fm}(R)}\otimes_R R^{\Gamma} $. The proof is complete.
\end{proof}

\begin{lemma}\label{Gamma parameter} Let $(R, \fm, k)$ be an equidimensional complete local ring that is $F$-rational on the punctured spectrum. Then for any sufficiently small choice of $\Gamma$ cofinite in $\Lambda$, $(\fq R^{\Gamma})^* = \fq^* R^{\Gamma}$ for all parameter ideals $\fq$ of $R$.
\end{lemma}

\begin{proof} By Lemmas \ref{Gamma F-rational} and \ref{Gamma local coho}, $R^{\Gamma}$ is $F$-rational on the punctured spectrum and $0^*_{H^d_{\fm}(R^{\Gamma})} = 0^*_{H^d_{\fm}(R)}\otimes_R R^{\Gamma} $ for any sufficiently small choice of $\Gamma$ cofinite in $\Lambda$. Since $R \to R^{\Gamma}$ is faithfully flat and $\fm R^{\Gamma}$ is the maximal ideal of $R^{\Gamma}$, we have $\ell_R(H^i_{\fm}(R)) = \ell_{R^{\Gamma}}(H^i_{\fm}(R^{\Gamma}))$ for all $i \le d-1$, and
$$\ell_R(0^*_{H^d_{\fm}(R)}) = \ell_{R^{\Gamma}}(0^*_{H^d_{\fm}(R^{\Gamma})}).$$
Let $\fq = (x_1, \ldots, x_d)$ be any parameter ideal of $R$. By Proposition \ref{upper above} there exists a positive integer $k$ such that
$$\ell_R(\fq_k^* / \fq_k) = \sum_{i=0}^{d-1}\binom{d}{i} \ell_R (H^i_{\fm}(R)) + \ell_R (0^*_{H^d_{\fm}(R)}),$$
where $\fq_k = (x_1^k, \ldots, x_d^k)$ is a parameter ideal of $R$. Since $\fq_k^* R^{\Gamma} \subseteq (\fq_k R^{\Gamma})^*$ and $R \to R^{\Gamma}$ is faithfully flat, we have $\ell_{R^{\Gamma}}((\fq_k R^{\Gamma})^*/\fq_k R^{\Gamma}) \ge \ell_R(\fq_k^* / \fq_k)$. Therefore
$$\ell_{R^{\Gamma}}((\fq_k R^{\Gamma})^*/\fq_k R^{\Gamma}) \ge \sum_{i=0}^{d-1}\binom{d}{i} \ell_{R^{\Gamma}} (H^i_{\fm}(R^{\Gamma})) + \ell_{R^{\Gamma}} (0^*_{H^d_{\fm}(R^{\Gamma})}).$$
By Proposition \ref{upper above} again we have
$$\ell_{R^{\Gamma}}((\fq_k R^{\Gamma})^*/\fq_k R^{\Gamma}) = \sum_{i=0}^{d-1}\binom{d}{i} \ell_{R^{\Gamma}} (H^i_{\fm}(R^{\Gamma})) + \ell_{R^{\Gamma}} (0^*_{H^d_{\fm}(R^{\Gamma})}).$$
Hence $\ell_{R^{\Gamma}}((\fq_k R^{\Gamma})^*/\fq_k R^{\Gamma}) = \ell_R(\fq_k^* / \fq_k)$, and so $(\fq_k R^{\Gamma})^* = \fq_k^* R^{\Gamma}$.\\
We next show that $(\fq R^{\Gamma})^* = \fq^* R^{\Gamma}$. Set $x = x_1 \ldots x_d$. By \cite[Proposition 3.3]{FH11} we have $\fq^* = \fq_k^* :_R x^{k-1}$ and $(\fq R^{\Gamma})^* =(\fq_k R^{\Gamma})^*:_{R^{\Gamma}} x^{k-1}$. Moreover $(\fq_k R^{\Gamma})^* = \fq_k^* R^{\Gamma}$ and $R \to R^{\Gamma}$ is a flat extension, so we obtain $(\fq R^{\Gamma})^* = \fq^* R^{\Gamma}$. This completes the proof.
\end{proof}

We now prove the main result of this paper.

\begin{proof}[Proof of the main theorem] Suppose $R$ is an $F$-injective local ring that is $F$-rational on the punctured spectrum. We will show that $\ell_R (\fq^* / \fq)$ does not depend on the choice of parameter ideal $\fq$. It is not hard to see that the $\fm$-adic complement $\widehat{R}$ is also an $F$-injective local ring that is $F$-rational on the punctured spectrum. Moreover $\ell_R (\fq^* / \fq) = \ell_{\widehat{R}} ((\fq \widehat{R})^* / \fq \widehat{R})$ (cf. \cite[Exercise 4.1]{H96}). Therefore we can assume henceforth that $R$ is complete. Let $\fq$ be any parameter ideal of $R$. By Proposition \ref{upper above} we need only to show that
$$\ell_R (\fq^* / \fq) = \sum_{i=0}^{d-1}\binom{d}{i} \ell_R (H^i_{\fm}(R)) + \ell_R (0^*_{H^d_{\fm}(R)}).$$
Taking a sufficiently small choice $\Gamma$ cofinite in $\Lambda$ satisfying Lemmas \ref{Gamma injec}, \ref{Gamma F-rational}, \ref{Gamma local coho}, and \ref{Gamma parameter}. We have that $R^{\Gamma}$ is an $F$-injective local ring that is $F$-rational on the punctured spectrum and $\fq^* R^{\Gamma} = (\fq R^{\Gamma})^*$ for all parameter ideals $\fq$. Since $R^{\Gamma}$ is $F$-finite,
$$\ell_{R^{\Gamma}}((\fq R^{\Gamma})^*/\fq R^{\Gamma}) = \sum_{i=0}^{d-1}\binom{d}{i} \ell_{R^{\Gamma}} (H^i_{\fm}(R^{\Gamma})) + \ell_{R^{\Gamma}} (0^*_{H^d_{\fm}(R^{\Gamma})})$$
by Theorem \ref{F-finite case}. However $\ell_{R^{\Gamma}}((\fq R^{\Gamma})^*/\fq R^{\Gamma}) = \ell_R(\fq^* / \fq)$, so
$$\ell_R (\fq^* / \fq) = \sum_{i=0}^{d-1}\binom{d}{i} \ell_R (H^i_{\fm}(R)) + \ell_R (0^*_{H^d_{\fm}(R)}).$$
The proof is complete.
\end{proof}

\section{Open Questions}
Recall that the {\it parameter test ideal} of $R$ is $\tau^R = \cap_{\fq} (\fq:_R \fq^*)$, where $\fq$ runs over all parameter ideals. Notice that $\tau^R$ is an $\fm$-primary ideal if and only if $R$ is $F$-rational on the punctured spectrum (so $R$ is generalized Cohen-Macaulay). Moreover any parameter ideal contained in $\tau^R$ is standard by \cite[Remark 5.11]{Hu98}. Based on Proposition \ref{upper above} we have the following natural question\footnote{I thank Professor Kei-ichi Watanabe for this question.}.
\begin{question} Let $(R, \fm, k)$ be an equidimensional excellent local ring that is $F$-rational on the punctured spectrum. Is it true that for every parameter ideal $\fq$ is contained in $\tau^R$ we have
$$\ell_R (\fq^* / \fq) = \sum_{i=0}^{d-1}\binom{d}{i} \ell_R (H^i_{\fm}(R)) + \ell_R (0^*_{H^d_{\fm}(R)}).$$
\end{question}

Let $A$ be an Artinian $R$-module with a Frobenius action $F: A \to A$. Then we define the {\it Frobenius closure} $0^F_A$ of the zero submodule of $A$ is the submodule of $A$ consisting all element $z$ such that $F^e(z) = 0$ for some $e \ge 0$. $0^F_A$ is the nilpotent part of $A$ by the Frobenius action. By \cite[Proposition 1.11]{HS77} and \cite[Proposition 4.4]{L97} there exists a non-negative integer $e$ such that $0^F_A = \ker (A \overset{F^e}{\longrightarrow} A)$ for all $i \ge 0$ (see also \cite{Sh07}). The smallest of such integers is called the {\it Harshorne-Speiser-Lyubeznik} number of $A$ and denoted by $HSL(A)$. We define the {\it Harshorne-Speiser-Lyubeznik} number of a local ring $(R, \frak m)$ as follows \footnote{The Harshorne-Speiser-Lyubeznik number of $R$ is often defined in terms only the top local cohomology module $H^d_{\frak m}(R)$. However, we think that it should be defined by all local cohomology modules. Moreover, our's definition is more suitable with $F$-singularities. Indeed, it is clear that $R$ is $F$-injective if and only if $HSL(R) = 0$.}
$$HSL(R): = \min \{ e \mid   0^F_{H^i_{\fm}(R)} =   \ker (H^i_{\fm}(R) \overset{F^e}{\longrightarrow} H^i_{\fm}(R)) \text{ for all } i = 0, \ldots, d\}.$$
 The $HSL(R)$ is closely related with the {\it Frobenius test exponent} of parameter ideals (see \cite{HKSY06, KS06}). Recall that the Frobenius test exponent of $R$, here we denote by $Fte(R)$, is the smallest non-negative integer $e$ such that $(\frak q^F)^{[p^e]} = \frak q^{[p^e]}$ for all parameter ideals $\frak q$, and $Fte(R) = \infty$ if we have no such $e$ \footnote{If $R$ is Cohen-Macaulay, then Katman and Sharp \cite{KS06} showed that $Fte(R)$ is just $HSL(R)$. In this paper, they posed the question that whether $Fte(R)$ is an integer for any local ring $R$. This question holds true for generalized Cohen-Macaulay rings, but it is still open in general. Recently, the author proved that $Fte(R) \ge HLS(R)$ for any local ring $R$.}.  The relation between $0^F_{H^i_{\fm}(R)}$ and $\fq^F$ appears explicit in \cite{QS17}. In more details, if $\fq^{F} = \fq$ for all parameter ideals $q$ then $0^F_{H^i_{\fm}(R)} = 0$ for all $i \ge 0$, i.e. $R$ is $F$-injective. Although the converse is not true for non-equidimensinal local rings, we believe the following question has an affirmative answer.

\begin{question} Let $(R, \fm, k)$ be an equidimensional excellent local ring of characteristic $p>0$. Then is it true that $\fq^{F} = \fq$ for all parameter ideals $\fq$, i.e. $Fte(R)=0$, if and only if $0^F_{H^i_{\fm}(R)} = 0$ for all $i \ge 0$, i.e. $HSL(R) = 0$.
\end{question}

Notice that Ma \cite{M15} gave a positive answer for the above question for generalized Cohen-Macaulay rings. Suppose $R$ is an excellent equidimension local ring that is $F$-injective on the punctured spectrum. Then we can check that $\ell_R \,(0^F_{H^i_{\fm}(R)}) < \infty$ for all $i \ge 0$. Inspired by the main result of this paper, we ask the following question.
\begin{question} Let $(R, \fm, k)$ be an excellent generalized Cohen-Macaulay local ring that is $F$-injective on the punctured spectrum. Is it true that
$$\ell_R (\fq^F / \fq) \le \sum_{i=0}^{d}\binom{d}{i} \ell_R (0^F_{H^i_{\fm}(R)})$$
for all parameter ideals $\fq$.
\end{question}
If $R$ is a generalized Cohen-Macaulay local ring (of characteristic $p>0$), then some power of $\fm$ is a standard ideal of $R$ (cf. Remark \ref{standard ideal} (2)). Hence there exists a non-negative integer $e$ such that every parameter ideal of $\fm^{[p^e]}$ is standard. This condition is equivalent to the condition that $F_*^e(R)$ is a Buchsbaum $R$-module provided $R$ is $F$-finite. I thank Nguyen Cong Minh for the following question.
\begin{question}
  Let $(R, \fm, k)$ be a generalized Cohen-Macaulay local ring of characteristic $p>0$. Does there exist an integer $e$ (that is bounded above by the Frobenius invariants of $R$ such as $HSL(R)$ and $Fte(R)$) such that $\fm^{[p^e]}$ is a standard ideal of $R$.
\end{question}


\begin{thebibliography}{99}

\bibitem{BMS16}
B. Bhatt, L. Ma and K. Schwede, \emph{The dualizing complex of $F$-injective and Du Bois singularities}, https://arxiv.org/abs/1512.05374.

\bibitem{BH98}

W. Bruns and J. Herzog, \emph{Cohen-Macaulay rings}, Cambridge University Press. {\bf39} (1998), revised edition.

\bibitem{CHL99} N.T. Cuong, N.T. Hoa and N.T.H. Loan, \emph{On centain length function associate to a system of
parameters in local rings}, Vietnam. J. Math. {\bf 27} (1999),
259--272.

\bibitem{CQ11}
N.T. Cuong and P.H. Quy, \emph{A splitting theorem for local cohomology and its applications}, J. Algebra {\bf331} (2011), 512--522.


\bibitem{EH08}
F. Enescu and M. Hochster, \emph{The Frobenius structure of local cohomology}, Algebra \& Number Theory {\bf 2}
(2008), 721--754.

\bibitem{F83}
R. Fedder, \emph{$F$-purity and rational singularity}, Trans. Amer. Math. Soc. {\bf 278} (1983), 461--480.

\bibitem{FH11} L. Fouli and C. Huneke, \emph{What is a system of parameters?}, Proc. Amer. Math. Soc. {\bf 139} (2011), 2681--2696.

\bibitem{G83}
S. Goto, \emph{On the Associated Graded Rings of
Parameter Ideals in Buchsbaum Rings }, J. Algebra {\bf 85} (1983),
490--534.

\bibitem{GN02} S. Goto and Y. Nakamura, \emph{The bound of the difference between parameter ideals and their tight closures}, Tokyo J. Math. {\bf 25} (2002), 41--48.

\bibitem{HS77} R. Hartshorne and R. Speiser, \emph{Local cohomological dimension in characteristicp}, Ann. of Math. {\bf 105} (1977), 45--79.

\bibitem{HH90}
M. Hochster and C. Huneke, \emph{Tight Closure, Invariant Theory, and the Brian\c{c}on-Skoda Theorem}, J. Amer. Math. Soc. {\bf3} (1990), 31--116.


\bibitem{HH94}
M. Hochster and C. Huneke, \emph{F-regularity, test elements, and smooth base change},
Trans. Amer. Math. Soc. {\bf346} (1994), 1--62.


\bibitem{H96}
C. Huneke,  \emph{Tight closure and its applications}, CBMS Lecture Notes in Mathematics,  Vol.\textbf{88}, Amer. Math. Soc., Providence, (1996).

\bibitem{Hu98} C. Huneke, \emph{Tight closure, parameter ideals, and geometry}, in: {\it Six Lectures on Commutative
Algebra} J. Elias, J.M. Giral, R.M. Mir\'{o}-Roig, S. Zarzuela
(ed.), Progress in Mathematics, vol. 166, Birkh\"{a}user Verlag,
Basel, 1998, 187--239.

\bibitem{HKSY06} C. Huneke, M. Katzman, R.Y. Sharp and Y. Yao, \emph{Frobenius test exponents for parameter ideals in generalized Cohen-Macaulay local rings},
J. Algebra {\bf 305} (2006), 516--539.

\bibitem{KS06} M. Katzman and R.Y. Sharp, \emph{Uniform behaviour of the Frobenius closures of ideals generated
by regular sequences}, J. Algebra {\bf 295} (2006) 231--246.

\bibitem{Ku69} E. Kunz, \emph{Characterizations of regular local rings of characteristic $p$}, Amer. J. Math. {\bf91} (1969), 772--784.

\bibitem{Ku76}
E. Kunz, \emph{On Noetherian rings of characteristic $p$}, Amer. J. Math. {\bf98} (1976), 999--1013.

\bibitem{L97} G. Lyubeznik, \emph{$F$-modules: applications to local cohomology and $D$-modules in characteristic $p>0$}, J. reine
angew. Math. {\bf 491} (1997), 65--130.

\bibitem{M14}
L. Ma, \emph{Finiteness property of local cohomology for $F$-pure local rings}, Int. Math. Research Notices ${\bf20}$ (2014),  5489--5509.


\bibitem{M15}
L. Ma, \emph{$F$-injectivity and Buchsbaum singularities}, Math. Ann. ${\bf362}$ (2015), 25--42.

\bibitem{Ma86}
H. Matsumura, \emph{Commutative ring theory}, Cambridge Studies in Advanced Mathematics $\bf{8}$, Cambridge University Press, Cambridge (1986).


\bibitem{QS17} P.H. Quy and K. Shimomoto, \emph{$F$-injectivity and Frobenius closure of ideals in Noetherian rings of characteristic $p>0$}, to appear in Adv. Math.,  https://arxiv.org/abs/1601.02524.

\bibitem{Sh07} R.Y. Sharp, \emph{On the Hartshorne-Speiser-Lyubeznik theorem about Artinian modules with a Frobenius action}, Proc. Amer. Math. Soc. {\bf 135} (2007), 665--670.

 \bibitem{Sm97}  K.E. Smith, \emph{$F$-rational rings have rational singularities}, Amer. J. Math. {\bf 119} (1997), 159--180.

\bibitem{SV86}
J. St\"{u}ckrad and W. Vogel, \emph{Buchsbaum rings and applications}, Springer-Verlag, Berlin, 1986.


\bibitem{Tr86}
N.V. Trung, \emph{Toward a theory of generalized Cohen-Macaulay modules,} Nagoya Math. J. {\bf 102} (1986), 1--49.

\bibitem{V95}
J.D. V\'{e}lez, \emph{Openness of the F-rational locus and smooth base change}, J. Algebra {\bf 172}
(1995), 425--453.

\end{thebibliography}
\end{document}